\documentclass[12pt]{article}
\usepackage{amsthm}
\usepackage{amsmath}
\usepackage[centerlast]{caption2}
\usepackage{amssymb}
\usepackage{latexsym}
\usepackage{graphicx}

\newtheorem{theorem}{Theorem}
\newtheorem{corollary}{Corollary}
\newtheorem{lemma}{Lemma}
\theoremstyle{remark}

\title{The  solution  of the  Brannan conjecture}
\author{Erhan Deniz,   Murat \c Caglar  and R\'obert Sz\'asz
}
\date{}
\begin{document}
\maketitle

\setlength{\textheight}{240mm} \setlength{\textwidth}{135mm}
\addtolength{\topmargin}{-12.5mm}

\def\ds{\displaystyle}
\theoremstyle{definition}

\begin{abstract}We make  the final  step to give a proof for the Brannan's conjecture. The basic tool of the study is a Mac-Laurin development and an adequately estimation of an integral.
\end{abstract}
\section{Introduction}
We consider the following   Mac-Laurin development:
\begin{equation}\label{x9ed7za}
\frac{(1+xz)^\alpha}{(1-z)^\beta}=\sum_{n=0}^\infty{A_n(\alpha,\beta,x)}z^n,
\end{equation}
where $\alpha>0, \ \beta>0, \ x=e^{i\theta}, \
\theta\in[-\pi,\pi],$ and   $z\in{U.}$  The radius of convergence
of the series  (\ref{x9ed7za})  is equal to $1.$ In \cite{5} the
author conjectured, that  if   $\alpha>0, \ \beta>0$  and
$|x|=1,$    then
$$|A_{2n-1}(\alpha,\beta,x)|\leq{A_{2n-1}(\alpha,\beta,1)},$$
where   $n$   is a natural number.
Partial results regarding this question were  already    proved in   \cite{1}, \cite{2}, \cite{5},    \cite{8}.
Concerning
the case  $\beta=1$,  and   $\alpha\in(0,1)$   partial     results    have been proved    in
\cite{3},  \cite{4}, \cite{6},  \cite{7}, \cite{9}.
\\Using the results from \cite{6}    and  \cite{9}    we will give  the solution of the  case   $\beta=1,$     $\alpha\in(0,1).$ 
 We  introduce  the    notation   $A_{2n-1}(\alpha,1,x)=A_{2n-1}(\alpha,x).$
The following theorems  have  been proved  in   \cite{6}    and  \cite{9}.
\begin{theorem}\cite{6} 
If $n$  is a natural number,  $n\leq26$     and    $\alpha\in(0,1),$   then the following inequality holds:
\begin{equation}\label{usxzs5n6sw}
|A_{2n-1}(\alpha,x)|\leq|A_{2n-1}(\alpha,1)|, \ \  \textrm{for  \  all} \  \  |x|=1.\nonumber
\end{equation}
\end{theorem}
\begin{theorem}  \cite{9}
\  If $n$  is a natural number,  $n\geq27$     and    $\alpha\in(0,1),$   then the following inequality holds:
\begin{equation}\label{usxzs5n6sw}
|A_{2n-1}(\alpha,x)|\leq|A_{2n-1}(\alpha,1)|, \ \  \textrm{for  \  all} \  |x|=1 \ \textrm{with} \ \arg(x)\in[-\frac{2\pi}{3},\frac{2\pi}{3}].\nonumber
\end{equation}
\end{theorem}

The aim of this paper  is to extend  this theorem to the case $\arg(x)\in[-\pi,\pi].$ \\ In order to prove our main result,  we need the lemmas  from the next   section.

 \section{Preliminaries}
\begin{lemma}
If $y\in[\frac{1}{2},1],$  then the following inequalities   hold
\begin{eqnarray}\label{b5x3yomm04d6z}
0.61>\frac{\pi}{3\sqrt{3}}\geq\sqrt{\frac{1-y}{1+y}}\Big(\arctan\sqrt{\frac{1-y}{1+y}}+\arctan{\frac{y}{\sqrt{1-y^2}}}\Big)\geq0.
\end{eqnarray}
\end{lemma}
\begin{proof}
We denote $t=\sqrt{\frac{1-y}{1+y}}.$    It is easily seen that
$y\in[\frac{1}{2},1]$    is equivalent to
 $t\in[0,\frac{1}{\sqrt{3}}],$   and we have
 $$\sqrt{\frac{1-y}{1+y}}\Big(\arctan\sqrt{\frac{1-y}{1+y}}+\arctan{\frac{y}{\sqrt{1-y^2}}}\Big)=t\left(\arctan{t}+\arctan\frac{1-t^2}{2t}\right).$$
Let $\varphi:[0,\frac{1}{\sqrt{3}}]\rightarrow\mathbb{R}$  be the function defined by
$\varphi(t)=t\Big(\arctan{t}+\arctan\frac{1-t^2}{2t}\Big).$
We have    $\varphi'(t)=\arctan{t}+\arctan\frac{1-t^2}{2t}-\frac{t}{1+t^2},$  and   $\varphi''(t)=\frac{-2}{(1+t^2)^2}.$  \\
Thus  $\varphi'$   is strictly decreasing and   $\frac{\pi}{2}=\lim_{\searrow0}\varphi'(t)\geq\varphi'(t)\geq\varphi'(\frac{1}{\sqrt{3}})=\frac{\pi}{3}-\frac{\sqrt{3}}{4}>0, \ t\in[0,\frac{1}{\sqrt{3}}].$   Consequently,   $\varphi$  is strictly increasing and     $$0.60459\ldots= \frac{\pi}{3\sqrt{3}}\geq\varphi(t)>0,  \    t\in(0,\frac{1}{\sqrt{3}}).$$
\end{proof}
In \cite{9}  the author proved the equality
\begin{eqnarray}\label{ccfpxo5n0m}(1+e^{i\theta})^\alpha=1+\frac{\alpha}{1!}e^{i\theta}+\frac{\alpha(\alpha-1)}{2!}e^{2i\theta}+\ldots+\frac{\alpha(\alpha-1)\ldots(\alpha-2n+2)}{(2n-1)!}e^{(2n-1)i\theta}\nonumber\\
+e^{2ni\theta}\frac{\alpha(\alpha-1)\ldots(\alpha-2n+1)}{(2n-1)!}\int_0^1(1-t)^{2n-1}(1+e^{i\theta}t)^{\alpha-2n}dt,
\ \theta\in(-\pi,\pi).\nonumber\end{eqnarray}
This equality is a basic tool in our study. Taking the absolute values on both sides, we infer
 \begin{eqnarray}\label{fkl0kfk3kk654}|A_{2n-1}(\alpha,e^{i\theta})|=\big|1+\frac{\alpha}{1!}e^{i\theta}\nonumber\\+\frac{\alpha(\alpha-1)}{2!}e^{2i\theta}+\ldots
+\frac{\alpha(\alpha-1)\ldots(\alpha-2n+2)}{(2n-1)!}e^{(2n-1)i\theta}\big|
=\big|(1+e^{i\theta})^\alpha \ \ \ \\-e^{2ni\theta}{\alpha(1-\alpha)(1-\frac{\alpha}{2})\ldots(1-\frac{\alpha}{2n-1})}\int_0^1(1-t)^{2n-1}(1+e^{i\theta}t)^{\alpha-2n}dt\big|,\nonumber\\
 \ \theta\in(-\pi,\pi).\nonumber\end{eqnarray}
We introduce the notation   $\beta_n=(1-\alpha)(1-\frac{\alpha}{2})\ldots(1-\frac{\alpha}{2n-1}).$\\
We will  prove an estimation for $A_{2n-1}(\alpha,1)$   in the followings.
\begin{lemma}
The following inequality holds
\begin{eqnarray}\label{ffdlkuzi}A_{2n-1}(\alpha,1)\geq2^\alpha-\frac{\alpha\beta_n}{2n}, \ \ \ \alpha\in(0,1).\end{eqnarray}
\end{lemma}
\begin{proof}
According to (\ref{fkl0kfk3kk654})   we have
\begin{eqnarray}\label{jzpwmdd4cc}A_{2n-1}(\alpha,1)=2^\alpha-{\alpha\beta_n}\int^1_0\Big(\frac{1-t}{1+t}\Big)^{2n-1}(1+t)^{\alpha-1}dt.\end{eqnarray}
The condition   $t,\alpha\in(0,1)$  implies   $(1+t)^{\alpha-1}\leq1$   and so it follows that
\begin{eqnarray}\label{12111msz}\int^1_0\Big(\frac{1-t}{1+t}\Big)^{2n-1}(1+t)^{\alpha-1}dt\leq\int^1_0\Big(\frac{1-t}{1+t}\Big)^{2n-1}dt.
\end{eqnarray}
The change of variable  $s=\frac{1-t}{1+t}$   gives
$\int^1_0\Big(\frac{1-t}{1+t}\Big)^{2n-1}dt=\int^1_0s^{2n-1}\frac{2}{(1+s)^2}ds.$
Chebyshev's inequality for monotonic functions claims:   if
$\varphi,\psi:[0,1]\rightarrow[0,\infty)$   are two integrable
functions of different monotony, then the following  integral  inequality
holds:
$$\int_0^1\varphi(s)\psi(s)ds\leq\int_0^1\varphi(s)ds\int_0^1\psi(s)ds.$$
Putting  $\varphi(s)=s^{2n-1},$  and   $\psi(s)=\frac{2}{(1+s)^2}$  in this inequality,  we infer

\begin{eqnarray}\label{tddt3}\int^1_0\Big(\frac{1-t}{1+t}\Big)^{2n-1}dt=\int^1_0s^{2n-1}\frac{2}{(1+s)^2}ds\leq\int^1_0s^{2n-1}ds\int^1_0\frac{2}{(1+s)^2}ds=\frac{1}{2n}.\end{eqnarray}
From (\ref{12111msz})  and (\ref{tddt3})  we deduce
\begin{eqnarray}\label{cskott3}\int^1_0\Big(\frac{1-t}{1+t}\Big)^{2n-\alpha}(1-t)^{\alpha-1}dt\leq\frac{1}{2n}.\end{eqnarray}
\end{proof}
Finally,  we infer   (\ref{ffdlkuzi})   from   (\ref{jzpwmdd4cc})  and   (\ref{cskott3}). 
\begin{lemma}
If  $\alpha\in(0,1),$   $y\in\big[\frac{1}{2},1\big]$  and
$n\in\mathbb{N}^*,$ then  the following  integral  inequality  holds
\begin{eqnarray}\label{uzuxuc}I_n(y)=
\int_0^1\Big(\frac{1-t}{\sqrt{1+t^2-2ty}}\Big)^{2n-1}\big(\sqrt{1+t^2-2ty}\big)^{\alpha-1}dt\nonumber\\ \leq\sqrt{\frac{1-y}{1+y}}\Big(\arctan\sqrt{\frac{1-y}{1+y}}
+\arctan{\frac{y}{\sqrt{1-y^2}}}\Big)-\frac{|2-2y|^{\frac{\alpha}{2}}}{\alpha}+\frac{1}{\alpha}.
\end{eqnarray}
\end{lemma}
\begin{proof}
The sequence $(I_n(y))_{n\geq1}, \ \
I_n(y)=\int_0^1\Big(\frac{1-t}{\sqrt{1+t^2-2ty}}\Big)^{2n-1}\big(\sqrt{1+t^2-2ty}\big)^{\alpha-1}dt$
is decreasing with respect to $n,$ and  consequently it is enough
to prove the inequality  (\ref{uzuxuc})  in case $n=1.$  We have
\begin{eqnarray}
I_n(y)\leq{I_1}(y)=\int_0^1\frac{1-t}{(1+t^2-2ty)^{1-\frac{\alpha}{2}}}dt=\int_0^1\frac{1-y}{(1+t^2-2ty)^{1-\frac{\alpha}{2}}}dt\nonumber\\
-\frac{1}{2}\int_0^1\frac{2(t-y)}{(1+t^2-2ty)^{1-\frac{\alpha}{2}}}dt\geq(1-y)\int_0^1\frac{1}{1-y^2+(t-y)^2}dt\nonumber\\
-\frac{1}{2}\int_0^1\frac{2(t-y)}{(1+t^2-2ty)^{1-\frac{\alpha}{2}}}dt=\sqrt{\frac{1-y}{1+y}}\arctan\frac{t-y}{\sqrt{1-y^2}}\Big|_0^1\nonumber\\
-\frac{(1+t^2-2ty)^{\frac{\alpha}{2}}}{\alpha}\Big|_0^1=\sqrt{\frac{1-y}{1+y}}\Big(\arctan\sqrt{\frac{1-y}{1+y}}+\arctan{\frac{y}{\sqrt{1-y^2}}}\Big)\nonumber\\
-\frac{|2-2y|^{\frac{\alpha}{2}}}{\alpha}+\frac{1}{\alpha},
\end{eqnarray}
and the proof is done.
\end{proof}
Lemma 1 and Lemma 3  imply the following result.
\begin{corollary}
If  $\alpha\in(0,1),$   $y\in\big[\frac{1}{2},1\big]$  and
$n\in\mathbb{N}^*,$ then  \ 
$\frac{\pi}{3\sqrt{3}}-\frac{|2-2y|^{\frac{\alpha}{2}}}{\alpha}+\frac{1}{\alpha}\geq
\int_0^1\Big(\frac{1-t}{\sqrt{1+t^2-2ty}}\Big)^{2n-1}\big(\sqrt{1+t^2-2ty}\big)^{\alpha-1}dt=I_n(y).
$
\end{corollary}
\section{Main Result}

\begin{theorem}
If $n$  is a natural number,  $n\geq27$     and
$\alpha\in(0,1)$   then the following inequality holds
\begin{equation}\label{uscsaxzs5n6sw}
{A_{2n-1}(\alpha,1)}\geq|A_{2n-1}(\alpha,e^{i\theta})|, \ \
\textrm{for  \  all} \
\theta\in[-\pi,-\frac{2\pi}{3}]\cup[\frac{2\pi}{3},\pi].
\end{equation}
\end{theorem}
\begin{proof}
We use equality (\ref{fkl0kfk3kk654})  again
\begin{eqnarray}\label{fkfk354}|A_{2n-1}(\alpha,e^{i\theta})|\nonumber\\=\big|1+\frac{\alpha}{1!}e^{i\theta}+\frac{\alpha(\alpha-1)}{2!}e^{2i\theta}+\ldots
+\frac{\alpha(\alpha-1)\ldots(\alpha-2n+2)}{(2n-1)!}e^{(2n-1)i\theta}\big|\\
\leq\frac{\alpha(1-\alpha)(2-\alpha)\ldots(2n-1-\alpha)}{(2n-1)!}\int_0^1(1-t)^{2n-1}\big|1+e^{i\theta}t\big|^{\alpha-2n}dt\nonumber\\+\big|1+e^{i\theta}
\big|^\alpha=\big|1+e^{i\theta}
\big|^\alpha+\alpha(1-\alpha)(1-\frac{\alpha}{2})(1-\frac{\alpha}{3})\ldots(1-\
\
\nonumber\\\frac{\alpha}{2n-1})\int_0^1(1-t)^{2n-1}\big|1+e^{i\theta}t\big|^{\alpha-2n}dt,
 \ \theta\in(-\pi,\pi).\nonumber\end{eqnarray}
Taking into  account    Lemma 2  and     (\ref{fkfk354}),   it follows that in order to prove   (\ref{uscsaxzs5n6sw})   we have to show that the following inequality holds
 \begin{eqnarray}\label{ttsjjkk}
2^\alpha-\frac{\alpha\beta_n}{2n}\geq\big|1+e^{i\theta}
\big|^\alpha+\alpha(1-\alpha)(1-\frac{\alpha}{2})(1-\frac{\alpha}{3})\ldots(1-\
\
\nonumber\\\frac{\alpha}{2n-1})\int_0^1(1-t)^{2n-1}\big|1+e^{i\theta}t\big|^{\alpha-2n}dt,
 \  \theta\in[-\pi,-\frac{2\pi}{3}]  \cup[\frac{2\pi}{3},\pi].
 \end{eqnarray}
 We denote $y=-\cos\theta,$  and   we get
\begin{eqnarray}\label{ta6astsjjkk}
|2-2y|^{\frac{\alpha}{2}}+\alpha{\beta_n}I_n(y)=\big|1+e^{i\theta}
\big|^\alpha+\alpha(1-\alpha)(1-\frac{\alpha}{2})(1-\frac{\alpha}{3})\ldots(1-\
\
\nonumber\\\frac{\alpha}{2n-1})\int_0^1(1-t)^{2n-1}\big|1+e^{i\theta}t\big|^{\alpha-2n}dt,
 \  \theta\in[-\pi,-\frac{2\pi}{3}]  \cup[\frac{2\pi}{3},\pi].
 \end{eqnarray}
Thus  inequality (\ref{ttsjjkk})  can be rewritten  in the following equivalent form
\begin{eqnarray}\label{x5wt3tsjvjkk}
2^\alpha\geq\frac{\alpha\beta_n}{2n}+\big(2-2y\big)^\frac{\alpha}{2}+\alpha\beta_nI_n(y).
  \  y\in[\frac{1}{2},1].
 \end{eqnarray}
  On the other hand,  according to Corollary 1,  we have  
\begin{eqnarray}\label{t3ts1jwlvn56zpx42jk2k}\big(2-2y\big)^\frac{\alpha}{2}(1-\beta_n)+\beta_n+\alpha\beta_n\Big(\frac{1}{2n}+\frac{\pi}{3\sqrt{3}}\Big)=\frac{\alpha\beta_n}{2n}+\big(2-2y\big)^\frac{\alpha}{2}\nonumber\\+\alpha\beta_n\Big(\frac{\pi}{3\sqrt{3}}-\frac{|2-2y|^{\frac{\alpha}{2}}}{\alpha}+\frac{1}{\alpha}\Big)\geq
\frac{\alpha\beta_n}{2n}+\big(2-2y\big)^\frac{\alpha}{2}+\alpha\beta_nI_n(y).
 \end{eqnarray}
The conditions  $y\in[\frac{1}{2},1]$   and  $n\geq27$   imply that  $\big(2-2y\big)^\frac{\alpha}{2}\leq1,$  and  $\frac{1}{2n}+\frac{\pi}{3\sqrt{3}}\leq\frac{65}{100}.$  Consequently we have
\begin{eqnarray}\label{h69990xjj1jv6zx4jk2k}1+\alpha\beta_n\frac{65}{100}\geq
\big(2-2y\big)^\frac{\alpha}{2}(1-\beta_n)+\beta_n+\alpha\beta_n\Big(\frac{1}{2n}+\frac{\pi}{3\sqrt{3}}\Big).
 \end{eqnarray}Since  $\ln2>0.65$   and  $\beta_n<1,$   we get  
\begin{eqnarray}\label{h6o9l9j90xjj1jv6zx4jk2k}2^\alpha\geq1+\alpha\ln2\geq1+\alpha\beta_n\frac{65}{100}.\end{eqnarray}
Finally,   inequalities  (\ref{t3ts1jwlvn56zpx42jk2k}),  (\ref{h69990xjj1jv6zx4jk2k})  and  (\ref{h6o9l9j90xjj1jv6zx4jk2k})
imply  (\ref{x5wt3tsjvjkk}),  and the proof is done.
\end{proof}
Summarizing 
Theorem 1,    Theorem  2  and  Theorem 3   imply the  following corollary.
\begin{corollary}
If  $x\in\mathbb{C}$     with   $|x|=1,$   then the inequality
$${A_{2n-1}}(\alpha,1)\geq|A_{2n-1}(\alpha,x)|$$
holds for every  $\alpha\in(0,1),$    and   $n\in\mathbb{N}^*.$
\end{corollary}

Murat \c Ca\u glar  and Erhan Deniz\\
Department of Mathematics\\
Faculty of Science and Letters\\
Kafkas  University\\
Kars\\
Turkey\\
e-mail: mcaglar25@gmail.com\\
e-mail : edeniz36@gmail.com\\
\newline
R\'obert Sz\'asz\\
 Department of Mathematics and Informatics\\
 Sapientia Hungarian University of Transylvania\\
 T\^argu-Mure\c s\\
 Romania\\
e-mail: rszasz@ms.sapientia.ro

\end{document}